\documentclass[12pt,a4paper, oneside, reqno]{amsart}

\usepackage[foot]{amsaddr} 

\usepackage{amsmath,amsthm,amssymb,bbm,bm}
\usepackage[bookmarksnumbered=false, colorlinks=true,citecolor=red,urlcolor=green,linkcolor=blue]{hyperref}
\usepackage{caption}
\usepackage{enumitem}
\usepackage{mleftright}
\usepackage{csquotes}
\usepackage{graphicx,float}
\usepackage{tikz}
\usepackage[a4paper,margin=2.5cm]{geometry}

\usepackage[labelformat=simple]{subcaption}

\usepackage{mathtools}
\mathtoolsset{showonlyrefs}

\theoremstyle{plain}
\newtheorem{theorem}{Theorem}[section]

\newtheorem{lemma}[theorem]{Lemma}
\newtheorem{proposition}[theorem]{Proposition}
\theoremstyle{definition}
\newtheorem{remark}[theorem]{Remark}

 \newcommand{\D}[1]{\mathop{\mathrm{d}#1}}
\DeclareMathOperator*{\conv}{conv}
\DeclareMathOperator*{\strh}{star}
\DeclareMathOperator*{\Per}{Per}
\DeclareMathOperator*{\Area}{Area}
\DeclareMathOperator*{\Si}{Si}
\DeclareMathOperator*{\Ci}{Ci}

\newcommand{\eps}{\varepsilon}
\newcommand{\sfP}{\mathsf{P}}
\newcommand{\sfA}{\mathsf{A}}

\newcommand{\bbE}{\mathbb{E}}
\newcommand{\bbR}{\mathbb{R}}
\newcommand{\bbP}{\mathbb{P}}
\newcommand{\bbC}{\mathbb{C}}
\newcommand{\bbD}{\mathbb{D}}

\newcommand{\cA}{\mathcal{A}}
\newcommand{\cC}{\mathcal{C}}
\newcommand{\cD}{\mathcal{D}}
\newcommand{\cE}{\mathcal{E}}
\newcommand{\cH}{\mathcal{H}}
\newcommand{\cK}{\mathcal{K}}

\newcommand{\cS}{\mathcal{S}}
\newcommand{\cV}{\mathcal{V}}
\newcommand{\bfW}{\boldsymbol{W}}
\bmdefine\be{\mathrm{e}}
\newcommand{\argmax}{\mathop{\mathrm{argmax}}}

\usepackage[draft]{todonotes} 
\title[Perimeter and area of Brownian convex hull upon exiting the unit disk]{Expected perimeter and area of the convex hull of planar Brownian motion stopped upon exiting the unit disk}

\author[R.\ Mrazovi\'{c}]{Rudi Mrazovi\'{c} $^{1}$}
\address{$^{1}$University of Zagreb Faculty of Science, Zagreb, Croatia}
\email{rudi.mrazovic@math.hr}

\author[H.\ Panzo]{Hugo Panzo $^{2}$}
\address{$^{2}$Department of Mathematics and Statistics, Saint Louis University, St.\ Louis, USA}
\email{hugo.panzo@slu.edu}

\author[S.\ \v{S}ebek]{Stjepan\ \v{S}ebek $^{3}$}
\address{$^{3}$University of Zagreb Faculty of Electrical Engineering and Computing, Zagreb, Croatia}
\email{stjepan.sebek@fer.unizg.hr}

\begin{document}

\begin{abstract}
We study the convex hull of planar Brownian motion run until the exit time from the unit disk. Our primary objectives are to compute the expected perimeter and expected area of this convex hull, thereby complementing recent results on the convex hull of reflecting Brownian motion in confined geometries. We reduce the problem of computing the expected perimeter to computing the expected value of the Brownian motion's maximum horizontal displacement at the exit time, and then recast this maximum in terms of harmonic measure in a domain we call the \emph{truncated disk}. The problem of computing the expected area is reduced to computing the expected value of the difference of squares of the Brownian motion's maximum horizontal displacement at the exit time, and the value of the vertical displacement at the time this maximum horizontal displacement is achieved. In particular, we obtain exact expressions for both the expected perimeter and the expected area. We conclude with further results on the expected areas of two related hulls of the Brownian path run until exiting the disk, namely, the \emph{star hull} and \emph{topological hull}.
\end{abstract}

\maketitle


\section{Introduction and Main Results}

Let $\bfW=(\bfW_t)_{t \geq 0} = (X_t, Y_t)_{t \geq 0}$ be the standard planar Brownian motion on $\bbR^2$ started at the origin $\boldsymbol{0}$. For a set $\cA \subset \bbR^2$, let $\conv \cA$ denote the convex hull spanned by $\cA$, that is, the intersection of all convex subsets of $\bbR^2$ containing $\cA$. For $t \geq 0$, let $\cH_t = \conv\bfW_{[0, t]}$ be the convex hull generated by a path of $\bfW$ run until time $t$. Denote the corresponding perimeter process by $\sfP_t = \Per(\cH_t)$, where $\Per(\cA)$ stands for the perimeter of the set $\cA \subset \bbR^2$, and the corresponding area process by $\sfA_t = \Area(\cH_t)$, where $\Area(\cA)$ stands for the area of the set $\cA \subset \bbR^2$. Moreover, let	$\bbD= \{\boldsymbol{x} \in \bbR^2 : \|\boldsymbol{x}\| < 1\}$ be the open unit disk and let
\begin{equation}\label{eq:disk_exit}
\tau_\bbD = \inf\{t \geq 0 : \bfW_t \notin \bbD\}
\end{equation}
denote the first exit time of $\bfW$ from $\bbD$. Our objects of interest are the random variables $\sfP_{\tau_{\bbD}}$ and $\sfA_{\tau_{\bbD}}$, namely, the perimeter and area of the convex hull of the trajectory of the Brownian motion $\bfW$ run until the first time it hits the boundary $\partial \bbD$; see Figure \ref{fig:convex-hull}. In particular, we want to compute the expected perimeter $\bbE[\sfP_{\tau_{\bbD}}]$ and expected area $\bbE[\sfA_{\tau_{\bbD}}]$.

\begin{figure}[ht]
  \centering
  \includegraphics[width=0.4\textwidth]{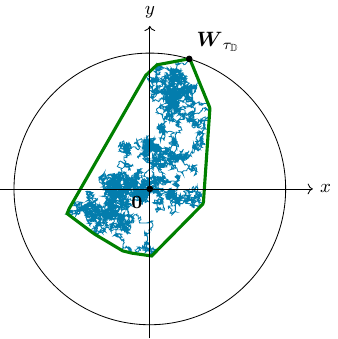}
  \caption{The convex hull (depicted in green) of a standard planar Brownian motion path starting at $\boldsymbol{0}$ and run until it exits the unit disk at $\bfW_{\tau_{\bbD}}$.}
  \label{fig:convex-hull}
\end{figure}
Similar questions have already been considered in the literature, but, to the best of our knowledge, not for paths of Brownian motion run until some random time. The problem of finding the expected value of the random variable $\sfP_1$ was proposed by Letac in 1978 \cite{Letac}, and then solved by Tak\'{a}cs in 1980 \cite{Takacs}. It holds that $\bbE[\sfP_1] = \sqrt{8\pi}$. When it comes to the expected area of the Brownian convex hull run up to time $1$, a result of El Bachir in \cite{ElBachir_PhD_thesis} shows that $\bbE[\sfA_1] = \pi/2$. These results were later generalized to the standard planar Brownian bridge \cite{goldman, Majumdar}, the union of independent standard planar Brownian motions \cite{Majumdar}, the union of independent standard planar Brownian bridges \cite{Majumdar}, and various combinations of independent standard planar Brownian motions and bridges \cite{sebek}. 

Another motivation for our work is a recent paper by De~Bruyne, B\'{e}nichou, Majumdar, and Schehr \cite{DeBruyneBenichouMajumdarSchehr}, where they considered Brownian motion confined to a $d$-dimensional ball with reflecting (Neumann) boundary conditions, and analyzed the statistics of the maximum of the trajectory along a fixed direction, as well as the growth of the convex hull. In the planar case $d = 2$, they obtained a precise large-time asymptotic for the mean perimeter $\bbE[\sfP_t]$ of the convex hull of a Brownian particle reflected at the boundary of a disk of radius $R$, showing that it converges to the perimeter $2\pi R$ of the circle with a stretched-exponential correction. Their predictions have been revisited and further developed in a rigorous probabilistic setting by Haas and Mallein \cite{HaasMallein}, via the analysis of a certain inhomogeneous fragmentation process. 

In the present paper we consider a complementary situation. Instead of a reflecting Brownian motion evolving for a long deterministic time inside the disk, we study planar Brownian motion killed when it first hits the boundary (Dirichlet boundary conditions). The exit time $\tau_{\bbD}$ is now an almost surely finite, random lifetime, and we are interested in the convex hull of the single excursion of the Brownian path from the center to the boundary of $\bbD$. Our main results give explicit representations for the expected perimeter and area of this ``Dirichlet Brownian hull''.

\begin{theorem}\label{thm:main}
Define the maximum horizontal displacement of the standard planar Brownian motion $\bfW=(\bfW_t)_{t \geq 0} = (X_t, Y_t)_{t \geq 0}$ at the time it first exits from the unit disk by
\begin{equation}\label{eq:horizontal_max}
M = \sup_{0 \leq t \leq \tau_\bbD} X_t.
\end{equation}
We have that the cumulative distribution function of $M$ is given by
	\begin{equation}\label{eq:cdf_M}
		\bbP(M < a) = \frac{2\arcsin a}{\pi - \arccos a}, \qquad 0<a<1.
	\end{equation}
	Furthermore,
	\begin{equation}\label{eq:exp_M}
		\bbE[M] = \int_0^{\pi/2}\left(1 - \frac{2t}{\frac{\pi}{2}+t}\right)\cos t\D{t} \approx 0.511655.
	\end{equation}
	Consequently, the expected perimeter of the convex hull of $\bfW$ when it first exits $\bbD$ is
	\begin{equation}\label{eq:exp_per}
		\bbE[\sfP_{\tau_{\bbD}}] = 2\pi \,\bbE[M] \approx 3.214826.
	\end{equation}
\end{theorem}

\begin{remark}\label{rem:Si_function}
Let $\Si(z)$ denote the \emph{sine integral}. This is the entire function defined by
	\begin{equation*}
		\Si(z) = \int_0^z \frac{\sin u}{u}\D u, \qquad z \in \mathbb{C};
	\end{equation*}
	see \cite[Equation 6.2.9]{DLMF}. We can use this special function to express our main result from Theorem \ref{thm:main} in a more compact way. More precisely, substituting $u=t+\frac{\pi}{2}$ in \eqref{eq:exp_M} results in
	\begin{equation*}
		\bbE[M] = \pi \big(\Si(\pi) - \Si(\pi/2)\big) - 1.
	\end{equation*}
	Therefore
	\begin{equation*}
		\bbE[\sfP_{\tau_{\bbD}}] = 2\pi^2 \big(\Si(\pi) - \Si(\pi/2)\big) - 2\pi.
	\end{equation*}
\end{remark}

\begin{theorem}\label{thm:area}
Define the almost surely unique random time at which the maximum horizontal displacement $M$ is attained by
\begin{equation}\label{eq:def_of_T}
    T = \argmax_{0 \le t \le \tau_\bbD} X_t.
\end{equation}
We show that
\begin{equation}\label{eq:exp_M2-YT2}
    \bbE[M^2 - Y_T^2] = 1 - \log 2 + \int_{\pi}^{2\pi} \frac{\cos u}{u} \D u \approx 0.210624.
\end{equation}
Consequently, the expected area of the convex hull of $\bfW$ when it first exits $\bbD$ is
	\begin{equation}\label{eq:exp_area}
		\bbE[\sfA_{\tau_{\bbD}}] = \pi \,\bbE[M^2 - Y_T^2] \approx 0.661695.
	\end{equation}
\end{theorem}

\begin{remark}\label{rem:Ci_function}
Let $\Ci(z)$ denote the \emph{cosine integral}. This is the principal value of the integral
	\begin{equation*}
		\Ci(z) = -\int_z^{\infty} \frac{\cos u}{u}\D u, \qquad z \in \mathbb{C};
	\end{equation*}
	see \cite[Equation 6.2.11]{DLMF}. We can use this special function to express our main result from Theorem \ref{thm:area} in a more compact way. More precisely,
	\begin{equation*}
		\bbE[\sfA_{\tau_{\bbD}}] = \pi \big(1 - \log 2 + \Ci(2\pi) - \Ci(\pi)\big).
	\end{equation*}
\end{remark}

Once we move away from the perimeter and area of the convex hull of standard planar Brownian motion, no more explicit results are known. Bounds for the expected diameter were considered in \cite{MX} and improved upon in \cite{Jovalekic}. Circumradius and inradius of the Brownian convex hull were studied in \cite{CPS}, and bounds for the expected values of those quantities were obtained. Many of these bounds were extended to Brownian motion in higher dimensions by \cite{hi_dim_hulls}. Another set of geometric functionals that are natural to study in higher dimensions are the so-called intrinsic volumes. The exact expressions for the expected value of all the intrinsic volumes of the convex hull spanned by standard $d$-dimensional Brownian motion run for unit time can be found in \cite{Eldan, Kabluchko-Zapor-TAMS, Molchanov-Wespi}.

The rest of the paper is organized as follows. In Section~\ref{sec:Cauchy} we present Cauchy's formula for the perimeter, and Blaschke's formula for the area (also known as Cauchy's area formula) for convex compact sets. Together with rotational invariance (of both the process and the domain), these formulas enable us to express the expected perimeter and expected area in simpler terms. More precisely, we express the expected perimeter in terms of the expected value of the Brownian motion's maximum horizontal displacement at the exit time. On the other hand, the problem of computing the expected area is reduced to computing the expected value of the difference of two squares:~that of the Brownian motion's maximum horizontal displacement at the exit time and its vertical displacement at the time this maximum horizontal displacement is achieved. 

A key idea in computing both expected values of interest is to recast the maximum horizontal displacement at the exit time in terms of harmonic measure in a domain we call the \emph{truncated disk}. In Section~\ref{sec:truncated-disk} we find the wanted harmonic measure by conformally mapping the truncated disk onto a wedge. In Section~\ref{sec:perimeter} we compute the closed formula for the expected perimeter, and in Section~\ref{sec:area} we do the same for the expected area. Finally, in Section~\ref{sec:star_topological} we comment on two related hulls of the Brownian path run until exiting the disk, namely, the \emph{star hull} and \emph{topological hull}.


\section{Cauchy's formulas and the support function of the Brownian convex hull}\label{sec:Cauchy}

For a unit vector $\be\in\bbR^2$, the support function of a convex compact set $\cK\subset\bbR^2$ is
\begin{equation*}
	h_\cK(\be) = \sup_{\boldsymbol{x}\in \cK} \langle \boldsymbol{x},\be \rangle,
\end{equation*}
where $\langle \cdot, \cdot \rangle$ denotes the standard scalar product on $\bbR^2$. For $\theta \in [0, 2\pi)$, we write
\begin{equation*}
	\be_{\theta} = (\cos \theta, \sin \theta), \qquad h_\cK(\theta) = h_\cK(\be_{\theta}) = \sup_{\boldsymbol{x}\in \cK} \langle \boldsymbol{x},\be_{\theta} \rangle.
\end{equation*}
In the special case when $\cK = \cH_{\tau_{\bbD}}$, that is, the convex hull of the Brownian motion $\bfW$ run until the exit time from the unit disk, we drop the subscript and write
\begin{equation}\label{eq:def_of_h_theta}
	h(\theta) = \sup_{\boldsymbol{x} \in \cH_{\tau_{\bbD}}} \langle \boldsymbol{x},\be_{\theta} \rangle = \sup_{0 \leq t \leq \tau_{\bbD}} \langle \bfW_t, \be_{\theta} \rangle.
\end{equation}

\subsection{Perimeter}
Using Cauchy's surface area formula \cite{cauchy1832memoire, TsukermanVeomett}, we have the representation
\begin{equation}\label{eq:cauchy_for_LtauD}
	\sfP_{\tau_{\bbD}} = \int_0^{2\pi} h(\theta) \D{\theta}.
\end{equation}
By rotational invariance of standard planar Brownian motion and of the unit disk $\bbD$, the process $(\langle \bfW_t, \be_\theta \rangle)_{0\leq t\leq \tau_{\bbD}}$ has the same law as $(X_t)_{0\leq t\leq \tau_{\bbD}}$. Hence, for all $\theta \in [0, 2\pi)$, it holds that
\begin{equation*}
	h(\theta) \stackrel{d}{=} \sup_{0\leq t\leq \tau_{\bbD}} X_t = M.
\end{equation*}
In particular,
\begin{equation*}
	\bbE[h(\theta)] = \bbE[M] \quad\text{for all } \theta \in [0, 2\pi).
\end{equation*}
Taking expectation in \eqref{eq:cauchy_for_LtauD} and applying Tonelli's theorem, we obtain
\begin{equation}
    \label{eq:exp_per_via_exp_M}
	\bbE[\sfP_{\tau_{\bbD}}] = \int_0^{2\pi} \bbE[h(\theta)]\D{\theta} = \int_0^{2\pi} \bbE[M]\D{\theta} = 2\pi\,\bbE[M],
\end{equation}
which is exactly the first equality in equation \eqref{eq:exp_per} in Theorem \ref{thm:main}. Thus, the problem of computing the expected perimeter of the Brownian convex hull in the disk is reduced to computing the expected value of the maximum horizontal displacement $M$. We compute much more than just the expected value of $M$. Namely, we develop an exact expression for the cumulative distribution function of $M$. In order to do this, as already mentioned, we first recast the law of $M$ in terms of harmonic measure in a particularly shaped domain that we call the truncated disk.

\subsection{Area}
Our starting point in computing the expected area will be Blaschke's area formula (also known as Cauchy's area formula); see \cite[Equation (2.19)]{Hsiung}. Using this formula, we have
\begin{equation*}
	\sfA_{\tau_{\bbD}} = \frac{1}{2} \int_0^{2\pi}\mleft(h(\theta)^2 - h'(\theta)^2\mright)\D{\theta},
\end{equation*}
where $h(\theta)$ is the support function defined in \eqref{eq:def_of_h_theta}. It is straightforward to see that the support function $h(\theta)$ is Lipschitz with deterministic constant $1$, so the derivative $h'(\theta)$ exists for almost every $\theta \in [0, 2\pi)$.
By using Fubini's theorem to compute the product measure of the set
\[
    \bigl\{ 
        (\theta, \omega) \in [0,2\pi) \times \Omega
        :
        h'(\theta) \text{ exists for the path } \boldsymbol{W}(\omega) 
    \big\}
\]
in two different ways, it follows from rotational invariance that for every $\theta \in [0,2\pi)$, the derivative $h'(\theta)$ exists almost surely.

Using again the rotational invariance of Brownian motion and of the unit disk $\bbD$, we have
\begin{equation}\label{eq:area_via_h0}
	\bbE[\sfA_{\tau_{\bbD}}] = \pi \,\bbE\mleft[h(0)^2 - h'(0)^2\mright].
\end{equation}
We now give the interpretation of $h'(\theta)$ exactly as it is done in \cite[Section 5.2]{Majumdar}. Let us write, for $\theta \in [0, 2\pi)$,
\begin{equation*}
	b_t(\theta) = \langle \bfW_t, \be_{\theta} \rangle = X_t \cos\theta + Y_t \sin\theta.
\end{equation*}
After taking derivatives with respect to $\theta$ we get
\begin{equation}\label{eq:bt-derivative}
    b_t'(\theta) = -X_t \sin\theta + Y_t \cos\theta = b_t(\theta+\pi/2).
\end{equation}
Hence, $b_t(\theta)$ and $b_t'(\theta)$ are two independent standard one-dimensional Brownian motions. Notice that we can write
\begin{equation*}
	h(\theta) = \sup_{0 \le t \le \tau_{\bbD}} b_t(\theta).
\end{equation*}
Moreover, we have
\[
    h(0)
    =
    b_T(0)
    =
    X_T
    =
    M,
\]
where $T$ (see \eqref{eq:def_of_T}) denotes the almost surely unique random time at which the maximum horizontal displacement is attained. Furthermore, we claim that, almost surely,
\begin{equation}
    \label{eq:h'(0)}
    h'(0)
    =b'_T(0)=
    Y_T.
\end{equation}
Let $T_n$ be the time at which the supremum defining $h(1/n)$ is attained.
The Brownian paths are almost surely continuous and have unique times that maximize $b_t(1/n)$ on the time interval $[0,\tau_{\bbD}]$,
so it follows that $T_n\to T$.
Equality \eqref{eq:h'(0)} now follows easily from the obvious inequalities
\[
    b_T(1/n)-b_T(0)
    \leq
    h(1/n)-h(0)
    \leq
    b_{T_n}(1/n)-b_{T_n}(0),
\]
dividing by $1/n \to 0$ (as $n \to \infty$) and using \eqref{eq:bt-derivative}.

Thus, the area formula \eqref{eq:area_via_h0} reduces to computing the expected value $\bbE[M^2 - Y_T^2]$, which gives \eqref{eq:exp_area} from Theorem~\ref{thm:area}. To determine this expected value, we find the joint distribution of random vector $(M, Y_T)$. In what follows, we identify $\bbR^2$ with $\bbC$ in the obvious way when it is convenient for us. 

\section{Harmonic measure of the truncated disk}
\label{sec:truncated-disk}

For $a \in (0, 1)$, define the truncated disk domain $\cD_a$ by
\begin{equation}\label{eq:truncated_disk}
\cD_a = \bbD \cap \{z \in \bbC :\Re z < a\}.
\end{equation}
Note that the boundary of $\cD_a$ can be naturally partitioned into two disjoint pieces: the closed vertical chord $\cV_a$ defined by
\begin{equation}\label{eq:chord}
\cV_a=\{z \in \bbC :|z|\leq 1,\, \Re z = a\},
\end{equation}
and the circular arc $\cC_a$ defined by
\begin{equation}\label{eq:arc}
\cC_a=\{z\in \bbC: |z|=1,\, \Re z<a\}.
\end{equation}
See Figure \ref{fig:truncated-disk} for a depiction of $\cD_a$ with $a=\frac{1}{2}$.

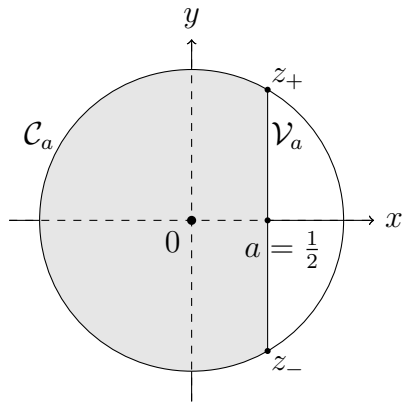
\begin{figure}[h]
\centering
\begin{tikzpicture}[scale=2]
  \draw[->] (-1.2,0) -- (1.2,0) node[right] {$x$};
  \draw[->] (0,-1.2) -- (0,1.2) node[above] {$y$};

  \def\a{0.5}
  \def\shift{0.1}
  \begin{scope}
    \clip (0,0) circle (1);
    \fill[gray!20] (-1.2,-1.2) rectangle (\a,1.2);
  \end{scope}

  \draw (0,0) circle (1);

  \draw (\a,{-(sqrt(1-\a*\a))}) -- (\a,{sqrt(1-\a*\a)});

  \fill (\a,0) circle (0.02);
  \node[below] at ({\a+\shift},0) {$a=\tfrac{1}{2}$};

  \fill (0,0) circle (0.03);
  \node[below left] at (0,0) {$0$};

  \draw[dashed,->] (-1.2,0) -- (1.2,0);
  \draw[dashed,->] (0,-1.2) -- (0,1.2);

  \node[above right] at (0.45,0.4) {$\cV_a$};
  \node[above left] at (-0.83,0.4) {$\cC_a$};
  
  \fill (\a,{-(sqrt(1-\a*\a))}) circle (0.02);
  \fill (\a,{(sqrt(1-\a*\a))}) circle (0.02);
  \node[below right] at (0.45,-0.83) {$z_a^-$};
  \node[above right] at (0.45,0.8) {$z_a^+$};
  
\end{tikzpicture}
\caption{The truncated disk domain $\cD_a$ with $a=\frac{1}{2}$.}
\label{fig:truncated-disk}
\end{figure}

Define the exit time of $\cD_a$ analogously to that of $\bbD$ and denote it by $\tau_{\cD_a}$, that is, 
\begin{equation*}
\tau_{\cD_a} = \inf\{t\geq 0 : \bfW_t \notin \cD_a\}.
\end{equation*}
Since $\cD_a\subset\bbD$, it is clear that $\tau_{\cD_a}\leq \tau_\bbD$. With this in mind and recalling the definition of $M$ from \eqref{eq:horizontal_max}, note that path continuity of $\bfW$ implies that
\begin{equation}
	\{M \geq a\}= \{\bfW_{\tau_{\cD_a}}\in\cV_a\}.\label{eq:chord_exit}
\end{equation}
In other words, $M\geq a$ if and only if $\bfW$ exits $\cD_a$ through the vertical chord $\cV_a$. 

For an arbitrary domain $\cD \subset \bbC$, let $\omega^z_\cD(\cdot)$ denote the harmonic measure on $\partial \cD$ seen from $z\in \cD$. More precisely, for any $z\in \cD$ and measurable $\cE\subset\partial \cD$, we have
\begin{equation*}
\omega^z_\cD(\cE) = \bbP(\bfW_{\tau_\cD} \in \cE \mid \bfW_0 = z).
\end{equation*}
We can use \eqref{eq:chord_exit} along with the above notation to write
\begin{equation}\label{eq:dist_of_M_as_harm_measure}
\bbP(M \geq a) = \omega^0_{\cD_a}(\cV_a).
\end{equation}
Motivated by this,
we first compute the total mass on $\cV_a$ with respect to $\omega_{\cD_a}^0$.
Later we shall need the finer exit-location density along this chord.

\subsection{A conformal map onto a wedge}

To determine the harmonic measure $\omega_{\cD_a}^z$ (and specifically $\omega_{\cD_a}^z(\cV_a)$),
we will use a conformal mapping $\mathfrak{l}_a$ from $\mathcal{D}_a$ to the wedge $\mathcal{W}_a$ -- see Lemma~\ref{l:Da-to-Wa} (most often, but not always, we will have $z=0$).
By conformal invariance of Brownian motion,
this will allow us to equate the desired harmonic measure of $\mathcal{V}_a \subset \partial \mathcal{D}_a$ seen from $z$ to that of $\mathfrak{l}_a(\mathcal{V}_a) \subset \partial \mathcal{W}_a$ seen from $\mathfrak{l}_a(z)$. 
Provided that $\mathfrak{l}_a$ extends continuously to the boundary of $\mathcal{D}_a$, 
we obtain
\begin{equation}
    \label{eq:conformal_invariance}
    \omega^z_{\mathcal{D}_a}(\mathcal{V}_a)=\omega^{\mathfrak{l}_a(z)}_{\mathcal{W}_a}\big(\mathfrak{l}_a(\mathcal{V}_a)\big);
\end{equation}
see \cite[Proposition 2.19]{Lawler} and \cite[Theorem 7.23]{MortersPeres}. The choice of $\mathcal{W}_a$ as the target domain of $\mathfrak{l}_a$ is justified by the simplicity of the appropriate $\mathfrak{l}_a$ and the fact that the harmonic measure on $\partial\mathcal{W}_a$ is already present in the literature, see e.g.\ \cite{Metzler}.

We now construct the desired conformal map $\mathfrak{l}_a$. Fix $a \in (0,1)$ and set
\begin{equation*}
    b_a = \sqrt{1-a^2}, \qquad z_a^\pm = a \pm i b_a.
\end{equation*}
These are the intersection points of the vertical line $\Re z = a$ with the unit circle. The vertical chord $\mathcal{V}_a$ on the boundary of $\mathcal{D}_a$ is the line segment with endpoints $z_a^-$ and $z_a^+$ (see Figure \ref{fig:truncated-disk}). Denote the Riemann sphere by $\hat{\bbC}$, and consider a M\"{o}bius transformation
\begin{equation}\label{eq:g_def}
\mathfrak{l}_a(z) = \frac{z - z_a^-}{z_a^+ - z}, \qquad z\in\hat{\bbC}.
\end{equation}
Recall that M\"{o}bius transformations are precisely the conformal automorphisms of the Riemann sphere.

We introduce two angle sizes that will be relevant in the rest of this section
\[
    \beta_a=\pi-\arccos a,
    \qquad
    \theta_a=2\arcsin a.
\]
\begin{lemma}
    \label{l:Da-to-Wa}
    Then $0<\theta_a<\beta_a<\pi$, and the conformal map  $\mathfrak{l}_a$ sends $\cD_a$ onto the wedge
    \[
        \mathcal{W}_a=\{r e^{i\theta}:r>0,\ 0<\theta<\beta_a\}.
    \]
    Moreover, $\mathcal{V}_a$ and $\overline{\mathcal{C}_a}$ are sent to the closed rays $[0,\infty]$ and $e^{i \beta_a}[0,\infty]$ inside $\hat{\bbC}$, respectively.
\end{lemma}

\begin{proof}
    M\"{o}bius transformations on $\hat{\bbC}$ map generalised circles (i.e.\ ordinary circles and lines) in $\hat{\bbC}$ to generalised circles. Consider the unit circle and the vertical line $x=a$. Since both of these sets pass through the points $z_a^{\pm}$, and the images of these points are $\mathfrak{l}_a(z_a^-)=0$ and $\mathfrak{l}_a(z_a^+)=\infty$, it follows that both of these sets are sent to lines through the origin. In the case of the vertical line, the fact that $\mathfrak{l}_a(a)=1 \in \mathbb{R}$ implies it is sent to the real axis. On the other hand, for the unit circle, put $\varphi=\arccos a$, so that $z_a^\pm=e^{\pm i\varphi}$.
    Then
    \[
        \mathfrak{l}_a(-1)
        =
        \frac{-1-e^{-i\varphi}}{e^{i\varphi}+1}
        =
        e^{i(\pi-\varphi)},
    \]
    and consequently the unit circle is sent to the line through the origin of slope $\beta_a = \pi - \varphi$.

    These two lines split $\hat{\bbC}$ into four connected regions, one of them being $\mathcal{W}_a$.
    Hence, all the wanted statements follow from the fact that $\mathfrak{l}_a(0) \in \mathcal{W}_a$,
    and this is the case since
    \begin{equation}
        \label{eq:la-0}
        \mathfrak{l}_a(0)=-\frac{z_a^-}{z_a^+}
             =-e^{-2i\varphi}
             =e^{i(\pi-2\varphi)}
             =e^{i\theta_a},
    \end{equation}
    and $0<\pi-2\varphi<\pi-\varphi=\beta_a$.
\end{proof}

A schematic depiction of the wedge $\mathcal{W}_a=\mathfrak{l}_a(\mathcal{D}_a)$ is shown in Figure \ref{fig:wedge} for $a=\frac{1}{2}$. We also record the image of the starting point $0$.

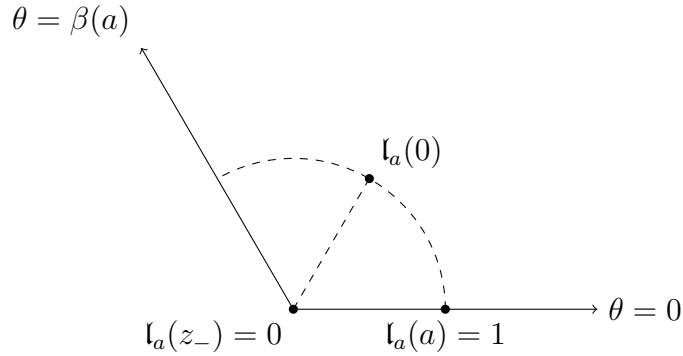
\begin{figure}[h]
\centering
\begin{tikzpicture}[scale=2]
  \draw[->] (0,0) -- (2,0) node[right] {$\theta = 0$};
  \draw[->] (0,0) -- ({-1},{sqrt(3)}) node[above left] {$\theta = \beta_a$};

  \fill (0,0) circle (0.03);
  \node[below left] at (0,0) {$\mathfrak{l}_a(z_a^-)=0$};

  \fill (1,0) circle (0.03);
  \node[below] at (1,0) {$\mathfrak{l}_a(a)=1$};

  \def\ang{60}
  \draw[dashed] (0,0) -- ({cos(\ang)},{sin(\ang)});
  \fill ({cos(\ang)},{sin(\ang)}) circle (0.03);
  \node[above right] at ({cos(\ang)},{sin(\ang)}) {$\mathfrak{l}_a(0)$};

  \draw[dashed] (1,0) arc (0:120:1);

\end{tikzpicture}
\caption{The wedge $\mathcal{W}_a = \mathfrak{l}_a(\mathcal{D}_a)$ with $a=\frac{1}{2}$ and $\beta_a=\frac{2\pi}{3}$.}
\label{fig:wedge}
\end{figure}

\subsection{\texorpdfstring{The probability of exiting $\mathcal{D}_a$ through $\mathcal{V}_a$}{The probability of exiting Da through Va}}

The total harmonic measure of either arm of a wedge follows from a simple
Dirichlet problem.

\begin{proposition}
    \label{p:dirichlet}
    Let $u$ be a point in $\mathcal{W}_a$.
    Then
    \[
        \omega_{\mathcal{W}_a}^{u}([0,\infty])
        = 
        1 - \frac{\arg u}{\beta_a},
    \]
    where the values of the argument function $\arg$ lie in $(-\pi,\pi]$.
    Consequently,
    for $z \in \cD_a$,
    \begin{equation}
        \label{eq:exit-Va}
        \omega_{\mathcal{D}_a}^z(\mathcal{V}_a)
        = 
        1 - \frac{\arg \mathfrak{l}_a(z)}{\beta_a}.
    \end{equation}
\end{proposition}

\begin{proof}
    It is well known that the function $v \mapsto \omega_{\mathcal{W}_a}^{v}([0,\infty])$ is the unique bounded harmonic function $h$ on $\mathcal{W}_a$ with boundary values $1$ on the lower arm and $0$ on the upper arm.
    It is easy to see that
    \[
        h(v)
        =
        1 - \frac{\arg v}{\beta_a}
        =
        1 - \frac{1}{\beta_a} \Im\log v,
    \]
    satisfies the properties above by using the fact that the imaginary (and real) part of a holomorphic function is harmonic. Now Equation \eqref{eq:exit-Va} follows from Lemma~\ref{l:Da-to-Wa} and conformal invariance \eqref{eq:conformal_invariance}.
\end{proof}

We note that the total harmonic mass of $\mathcal{V}_a$, provided by the previous proposition, is all that is needed to determine the law of $M$, and therefore the expected perimeter through \eqref{eq:exp_per_via_exp_M}.
For the area calculation, however, we need to know where on $\mathcal{V}_a$ the process exits.

\subsection{\texorpdfstring{The exit-location density on $\mathcal{V}_a$}{The exit-location density on Va}}

Let
\[
    \lambda_a=\frac{\pi}{\beta_a},
    \qquad
    \Phi_a=\lambda_a\theta_a
           =\frac{2\pi\arcsin a}{\pi-\arccos a}.
\]
Notice that $1<\lambda_a<2$ and $0<\Phi_a<\pi$.

\begin{proposition}
    \label{p:Va-density}
    For $y\in (-b_a,b_a)$,
    density of $\omega_{\cD_a}^0$ at $a+i\,\D y$ is
    \begin{equation}
        \label{eq:Va-density}
        \omega_{\cD_a}^0(a+i\,\D y)
        =
        \frac{2b_a}{\beta_a(b_a-y)^2}
        \frac{\rho_a(y)^{\lambda_a-1}\sin\Phi_a}{\sin^2\Phi_a + \bigl(\rho_a(y)^{\lambda_a}-\cos\Phi_a\bigr)^2} \, \D y,
    \end{equation}
    where $\rho_a(y):= \mathfrak{l}_a(a+iy) = (b_a+y)/(b_a-y)$.
\end{proposition}

\begin{proof}
    Under $\mathfrak{l}_a$, 
    the point $0$ is sent to $e^{i\theta_a}$ (see \eqref{eq:la-0}) and the vertical chord is sent to the lower arm of the wedge $\mathcal W_a$.
    The result of Metzler \cite[Corollary~2.2]{Metzler}, 
    applied with $\alpha=\beta_a$, $r_0=1$, and $\theta_0=\theta_a$, gives the unconditioned exit density on that arm
    \[
        \omega_{\mathcal{W}_a}^{e^{i\theta_a}}(\D r)
        =
        \frac{1}{\beta_a}
        \frac{r^{\lambda_a-1}\sin\Phi_a}{\sin^2\Phi_a+(r^{\lambda_a}-\cos\Phi_a)^2}\,\D r.
    \]
    To obtain \eqref{eq:Va-density} one simply uses conformal invariance of harmonic measure and the change of variables $r=\rho_a(y)$.
\end{proof}


\section{\texorpdfstring{The distribution of $M$ and the expected perimeter}{The distribution of M and the expected perimeter}}\label{sec:perimeter}

We are now ready to prove Theorem \ref{thm:main}.

\begin{proof}[Proof of Theorem \ref{thm:main}]
Proposition~\ref{p:dirichlet} and \eqref{eq:dist_of_M_as_harm_measure} give
\[
    \bbP(M \ge a) 
    = 
    1 - \frac{1}{\pi - \arccos a}\arg \frac{-z_a^-}{z_a^+}
    =
    1 - \frac{2\arcsin a}{\pi - \arccos a}.
\]
Therefore,
\begin{equation*}
\bbP(M < a) = \frac{2\arcsin a}{\pi - \arccos a}
	\end{equation*}
and this proves formula \eqref{eq:cdf_M} for the cumulative distribution function of $M$.
We can now use this formula to derive the integral representation \eqref{eq:exp_M} for $\bbE[M]$. Since $M\in[0,1]$ almost surely, we have
\begin{align}
\bbE[M] &= \int_0^1 \mathbb{P}(M \ge a)\D{a}\nonumber\\
& = \int_0^1\left(1 - \frac{2\arcsin a}{\pi - \arccos a}\right)\D{a}.\label{eq:exp_M-first}
\end{align}
Set $a=\sin t$ with $t\in[0,\pi/2]$. Then we get
\begin{equation*}
	\arcsin a = t,\quad \arccos a = \frac{\pi}{2} - t,\quad \D a = \cos t\, \D t.
\end{equation*}
Substituting these into \eqref{eq:exp_M-first} yields
\begin{equation*}
	\bbE[M] = \int_0^{\pi/2}\left(1 - \frac{2t}{\tfrac{\pi}{2}+t}\right)\cos t\D{t},
\end{equation*}
which is precisely \eqref{eq:exp_M}. The integral \eqref{eq:exp_M} does not appear to simplify to a familiar constant, however, see Remark \ref{rem:Si_function} for an alternative expression. Numerical integration yields
\begin{equation*}
	\bbE[M]\approx 0.511655.
\end{equation*}
Combining \eqref{eq:exp_per_via_exp_M} and \eqref{eq:exp_M} we clearly get \eqref{eq:exp_per}, and the numerical value
\begin{equation*}
	\bbE[\sfP_{\tau_{\bbD}}] \approx 3.214826.
\end{equation*}
This completes the proof of Theorem \ref{thm:main}.
\end{proof}

\begin{remark}
	It is interesting to notice that, as a trivial corollary of Theorem \ref{thm:main}, we have that for $a = 1/2$, $\bbP(M \ge a) = 1/2$. Namely, when the disk is truncated precisely at $a = 1/2$, we get that the probability of exiting through the vertical chord $\cV_a$ is the same as the probability of exiting through the circular arc $\cC_a$.
\end{remark}


\section{\texorpdfstring{The distribution of $(M,Y_T)$ and the expected area}{The distribution of (M,YT) and the expected perimeter}}\label{sec:area}

Before proving Theorem~\ref{thm:area}, we find a joint distribution of $(M, Y_T)$.

\begin{proposition}
    \label{p:mainmarkov}
    Let
    \begin{equation}
        \label{eq:eta-conv}
        q_\eps(a,y)
        = 
        \frac{\omega^{a+iy}_{\cD_{a+\eps}}(\cC_{a+\eps})}{\eps}
        \qquad\text{and} \qquad
        \eta_a(y)
        =
        \lim_{\eps \downarrow 0}
            q_\eps(a,y).
    \end{equation}
    Then for every bounded continuous function $f$,
    \[
        \bbE[f(M,Y_T)]
        =
        \int_0^1
            \int_{-b_a}^{b_a}
                f(a,y)\eta_a(y) \,
                \omega^0_{\mathcal{D}_a}(a+i\,\D y) \, \D a.
    \]
\end{proposition}

\begin{proof}
    Note that $T$ is \emph{not} a stopping time,
    so we first use the strong Markov property to prove the following identity for the closely related  probabilities
    \begin{equation}
        \label{eq:pom1}
        \bbP\bigl(
            M\in [a,a+\eps),\ Y_{\tau_{\cD_a}}\in B
        \bigr)
        =
        \int_B
            \omega^{a+iy}_{\cD_{a+\eps}}(\cC_{a+\eps})
            \,\omega^0_{\cD_a}(a+i\,\D y),
    \end{equation}
    where $a\in (0,1)$, $\eps > 0$, and $ B \subset [-b_a, b_a]$.

    Consequently, for every bounded function $g$,
    \begin{equation}
        \label{eq:pom2}
        \bbE\bigl[
            g(Y_{\tau_{\cD_a}})
            1_{M\in [a,a+\eps)}
        \bigr]
        =
        \int
            g(y)\,
            \omega^{a+iy}_{\cD_{a+\eps}}(\cC_{a+\eps})
            \,\omega^0_{\cD_a}(a+i\,\D y).
    \end{equation}

    We use the strong Markov property at the first exit time $\tau_{\cD_a}$ of $\bfW$ from $\cD_a$ (see e.g.\ \cite[Remark~2.17]{MortersPeres}).
    For any event $E$ measurable up to time $\tau_{\cD_a}$ and bounded function $F$ on path space,
    \begin{equation}
        \label{eq:markov}
        \bbE\bigl[
            1_E \, F(\theta_{\tau_{\cD_a}}\bfW) 
        \bigr]
        =
        \bbE\bigl[
            1_E \, \bbE_{\bfW_{\tau_{\cD_a}}}[F(\bfW)]
        \bigr],
    \end{equation}
    where $\theta_{\tau_{\cD_a}}$ denotes the time shift by $\tau_{\cD_a}$.

    Take
    \[
        E=\{\bfW_{\tau_{\cD_a}}\in a+iB\},
    \]
    and note that, on $E$, 
    \[
        M \in [a,a+\eps)
        \qquad \text{if and only if} \qquad
        \tau_{\bbD} \circ \theta_{\tau_{\cD_a}} < \tau_{\cD_{a+\eps}} \circ \theta_{\tau_{\cD_a}}.
    \]
    Hence, with
    \[
        F(\bfW)
        =
        1_{\tau_{\bbD}<\tau_{\cD_{a+\eps}}},
    \]
    the strong Markov property \eqref{eq:markov} gives
    \begin{align*}
        \bbP\bigl(M\in [a,a+\eps),\ Y_{\tau_{\cD_a}}\in B\bigr)
        &=
        \bbE\bigl[
            1_E \,
            \bbE_{\bfW_{\tau_{\cD_a}}}[F(\bfW)]
        \bigr].
    \end{align*}
    By the definition of harmonic measure on $\cD_a$, the right-hand side is
    \begin{align*}
        \int_B
            \bbE_{a+iy}[F(\bfW)] \,
            \omega^0_{\cD_a}(a+i\,\D y)
        &=
        \int_B
            \bbP_{a+iy}(\tau_{\bbD}<\tau_{\cD_{a+\eps}}) \,
            \omega^0_{\cD_a}(a+i\,\D y) \\
        &=
        \int_B
            \omega^{a+iy}_{\cD_{a+\eps}}(\cC_{a+\eps}) \,
            \omega^0_{\cD_a}(a+i\,\D y).
    \end{align*}
    This proves \eqref{eq:pom1},
    and \eqref{eq:pom2} follows immediately.

    It remains to pass from $Y_{\tau_{\cD_a}}$ to $Y_T$.
    The crucial idea is to exploit the obvious fact that, for a fixed path of $\bfW$,
    as $a \uparrow M$, we have that $\tau_{\cD_a} \uparrow T$ and hence $Y_{\tau_{\cD_a}} \to Y_T$.
    To formally use this observation,
    let $K_n$ be the unique integer such that 
    \[
        M \in [K_n/n,(K_n+1)/n).
    \]
    By the pathwise observation above, $Y_{\tau_{\cD_{K_n/n}}} \to Y_T$,
    and hence for every bounded continuous function $f$,
    dominated convergence gives
    \[
        \bbE[f(M,Y_T)]
        =
        \lim_{n\to\infty}
            \bbE\bigl[
                f(K_n/n,Y_{\tau_{\cD_{K_n/n}}})
            \bigr].
    \]
    Conditioning on the value of $K_n$, and using \eqref{eq:pom2}, we obtain
    \begin{align*}
        \bbE\bigl[
            f(K_n/n,Y_{\tau_{\cD_{K_n/n}}})
        \bigr]
        &=
        \sum_{k=0}^{n-1}
        \bbE\bigl[
            f(k/n,Y_{\tau_{\cD_{k/n}}})
            1_{\{k/n\leq M<(k+1)/n\}}
        \bigr] \\
        &=
        \sum_{k=0}^{n-1}
            \int_{-b_{k/n}}^{b_{k/n}}
                f(k/n,y) \,
                \omega^{k/n+iy}_{\cD_{(k+1)/n}}(\cC_{(k+1)/n}) \,
                \omega^0_{\cD_{k/n}}(k/n+i\,dy).
    \end{align*}

    We conclude that
    \[
        \bbE[f(M,Y_T)]
        =
        \lim_{n\to \infty}
            \sum_{k=0}^{n-1}
                \frac1n
                \int_{-b_{k/n}}^{b_{k/n}}
                    f(k/n,y) \,
                    q_{1/n}(k/n,y)\,
                    \omega^0_{\cD_{k/n}}(k/n+i\,dy),
    \]
    and it remains to prove that the right hand side converges to
    \[
        \int_0^1
            \int_{-b_a}^{b_a}
                f(a,y)\eta_a(y) \,
                \omega^0_{\mathcal{D}_a}(a+i\,\D y) \, \D a.
    \]
    We use the change of variables $r=\rho_a(y) \in (0,\infty)$ from Proposition \ref{p:Va-density},
    whose inverse is
    \[  
        y_a(r) := \frac{b_a(r-1)}{r+1}.
    \]
    Note that in the $r$-variable,
    \[
        \omega^0_{\cD_a}(a+i\,\D y)
        =
        p(a,r) \, \D r,
    \]
    where
    \[
        p(a,r)
        =
        \frac{\lambda_a r^{\lambda_a-1}\sin\Phi_a}{\pi\bigl(r^{2\lambda_a}-2r^{\lambda_a}\cos\Phi_a + 1\bigr)}.
    \]
    Using Proposition~\ref{p:dirichlet}, we have
    \begin{equation}\label{eq:exit_through_arc}
        q_\eps(a,y) = \frac{\arg \left( \mathfrak{l}_{a+\eps}(a+iy) \right)}{\eps\, \beta_{a+\eps}}.
    \end{equation}
    It is straightforward to compute that
    \begin{equation*}
        \mathfrak{l}_{a+\eps}(a+iy)
         = \frac{1-(a+\eps)^2-\eps^2-y^2 + 2\eps i \sqrt{1 - (a+\eps)^2}}{\eps^2 + (\sqrt{1 - (a+\eps)^2} - y)^2}.
    \end{equation*}
    Hence
    \begin{equation}\label{eq:arg_la}
        \arg \left( \mathfrak{l}_{a+\eps}(a+iy) \right) = \arctan \left( \frac{2\eps\sqrt{1 - (a+\eps)^2}}{1-(a+\eps)^2-\eps^2-y^2} \right).
    \end{equation}
    Plugging \eqref{eq:exit_through_arc} and \eqref{eq:arg_la} into \eqref{eq:eta-conv} and using $\lim_{x\to 0} (\arctan x)/x = 1$ gives
    \begin{equation*}
        \eta_a(y) = \frac{2b_a}{\beta_a (b_a^2 - y^2)}.
    \end{equation*}

    Fix some small $\delta>0$.
    It is easy to see, using \eqref{eq:exit_through_arc} and \eqref{eq:arg_la},
    that for all small $\eps>0$ and $a \in [\delta,1-\delta]$,
    \[
        q_\eps(a,y) 
        =
        O_\delta\mleft(\frac1{b_a^2-y^2}\mright),
    \]
    that is, in the $r$-coordinate,
    \[
        q_\eps(a,y_a(r)) 
        =
        O_\delta\mleft(\frac{(1+r)^2}{r}\mright).
    \]
    Letting $\underline{\lambda} = \inf_{a \in [\delta,1-\delta]} \lambda_a > 1$,
    we obtain an integrable bound
    \begin{equation}
        \label{eq:integrable-bound}
        q_\eps(a,y_a(r)) \, p(a,r)
        = \begin{cases}
            O_\delta(r^{\underline{\lambda}-2}), & 0<r\leq 1, \\
            O_\delta(r^{-\underline{\lambda}}), & r > 1.
        \end{cases}
    \end{equation}
    Note that convergence
    \[
        q_{1/n}(a,y_a(r)) \to \eta_a(y_a(r))
    \]
    is uniform on $[\delta,1-\delta] \times [1/R,R]$ for arbitrary $R>0$.
    Therefore the usual Riemann sums converge on this set,
    and the integrable bound \eqref{eq:integrable-bound} allows us to let $R \to \infty$ to conclude
    \begin{align*}
        &\lim_{n\to \infty}
            \sum_{\delta n\leq k \leq (1-\delta)n}
                \frac1n
                \int_{-b_{k/n}}^{b_{k/n}}
                    f(k/n,y) \,
                    q_{1/n}(k/n,y)\,
                    \omega^0_{\cD_{k/n}}(k/n+i\,\D y)\\
        &\qquad \qquad=
        \int_\delta^{1-\delta}
            \int_{-b_a}^{b_a}
                f(a,y)\eta_a(y) \,
                \omega^0_{\mathcal{D}_a}(a+i\,\D y) \, \D a.
    \end{align*}
    Hence, it remains to remove the cutoff $\delta\to 0$.
    Indeed, it is easy to see that these endpoint contributions are negligible.
    On the discrete side, they are bounded by
    \[
        \|f\|_\infty \bbP(M \notin (\delta+1/n,1-\delta)),
    \]
    which becomes negligible as $\delta \to 0$ and $n\to \infty$.
    On the limiting-integral side,
    the contribution is bounded by
    \begin{equation}\label{eq:bound_integral-side}
        \|f\|_\infty 
        \int_{-b_a}^{b_a}
                \eta_a(y) \,
                \omega^0_{\mathcal{D}_a}(a+i\,\D y)
        = \|f\|_\infty \frac{\lambda_a^2}{\pi b_a},
    \end{equation}
    which is integrable on $(0,1)$, and hence negligible outside $[\delta,1-\delta]$. The precise computation of the integral in~\eqref{eq:bound_integral-side} is quite technical, but completely analogous to the computation of the integral in~\eqref{eq:area_int_step1} which appears just below, and where we provide more details.
\end{proof}

With the distribution of $(M,Y_T)$ in hand,
it is now straightforward to calculate $\bbE[\sfA_{\tau_\bbD}]$.
\begin{proof}[Proof of Theorem \ref{thm:area}]
    Using the notation from Section \ref{sec:truncated-disk},
    Proposition~\ref{p:Va-density} gives
    \[
        \omega_{\cD_a}^0(a+i\,\D y)
        =
        \frac{2b_a\lambda_a \rho_a(y)^{\lambda_a-1} \sin\Phi_a}{\pi(b_a-y)^2 (\rho_a(y)^{2\lambda_a} - 2\rho_a(y)^{\lambda_a}\cos\Phi_a + 1)}\, \D y.
    \]
    By Proposition \ref{p:mainmarkov},
    \begin{align}\label{eq:area_int_step1}
        & \bbE [M^2 - Y_T^2] \nonumber \\
        & \!=\! \int_0^1\! \int_{-b_a}^{b_a} (a^2 - y^2) \frac{2b_a}{\beta_a(b_a^2 - y^2)} \frac{2b_a\lambda_a\rho_a(y)^{\lambda_a-1} \sin \Phi_a}{\pi(b_a-y)^2(\rho_a(y)^{2\lambda_a} - 2\rho_a(y)^{\lambda_a} \cos\Phi_a + 1)} \D y \D a \nonumber \\
        & \!=\! \int_0^1 \!\int_0^{\infty} \mleft((2a^2-1)(\rho_a(y)^2+1) + 2\rho_a(y)\mright) \frac{\lambda_a^2\rho_a(y)^{\lambda_a-2} \sin\Phi_a}{2\pi^2 b_a (\rho_a(y)^{2\lambda_a} - 2\rho_a(y)^{\lambda_a} \cos\Phi_a + 1)} \D{\rho_a(y)} \D a \nonumber\\
        & \!=\! \int_1^2 \frac{\sin \Phi_a}{2\pi} \int_0^{\infty} \frac{\rho_a(y)^{\lambda_a-2}[\cos(2\pi/\lambda_a)(\rho_a(y)^2+1) + 2\rho_a(y)]}{\rho_a(y)^{2\lambda_a} - 2\rho_a(y)^{\lambda_a} \cos\Phi_a + 1} \D{\rho_a(y)} \D{\lambda_a},
    \end{align}
with $\Phi_a = \lambda_a \theta_a = 2\pi - \pi\lambda_a$, where in the second line we used the relation
\[
y = b_a(\rho_a(y)-1)/(\rho_a(y)+1),
\]
and in the third line we used the relation $a = -\cos(\pi/\lambda_a)$. We first evaluate the inner integral. Denote by
    \begin{equation}\label{eq:def_of_J}
        J(s) = \int_0^{\infty} \frac{\rho_a(y)^{s-1}}{\rho_a(y)^{2\lambda_a} - 2\rho_a(y)^{\lambda_a} \cos\Phi_a + 1} \D{\rho_a(y)}.
    \end{equation}
    In terms of the function $J(s)$, the inner integral in \eqref{eq:area_int_step1} is
    \begin{equation}\label{eq:inner_int_in_terms_of_J}
        \cos\left( \frac{2\pi}{\lambda_a} \right) [J(\lambda_a + 1) + J(\lambda_a - 1)] + 2J(\lambda_a).
    \end{equation}
    Using the substitution $r = \rho_a(y)^{\lambda_a}$ in \eqref{eq:def_of_J} and then \cite[Eq.~3.252.12]{Gradshteyn_Ryzhik} with $\mu = s/\lambda_a$, $a = 1$, and $t = \pi - \Phi_a$, we get
    \begin{equation*}
        J(s) = \frac{\pi}{\lambda_a \sin\Phi_a} \frac{\sin \left( (1-s/\lambda_a) (\pi - \Phi_a) \right)}{\sin(\pi s / \lambda_a)}.
    \end{equation*}
    It is now straightforward to see that
    \begin{equation*}
        J(\lambda_a-1) = \frac{\pi}{\lambda_a \sin\Phi_a}, \quad J(\lambda_a+1) = \frac{\pi}{\lambda_a \sin\Phi_a}, \quad \text{and} \quad J(\lambda_a) = \frac{\pi(\lambda_a-1)}{\lambda_a \sin\Phi_a}.
    \end{equation*}
    Hence,
    \begin{equation*}
        \cos\left( \frac{2\pi}{\lambda_a} \right) [J(\lambda_a + 1) + J(\lambda_a - 1)] + 2J(\lambda_a) = \frac{2\pi}{\lambda_a \sin\Phi_a} \left[ \cos\left( \frac{2\pi}{\lambda_a} \right) + \lambda_a - 1 \right].
    \end{equation*}
    Plugging this into \eqref{eq:area_int_step1} we get
    \begin{align*}
        \bbE[M^2 - Y_T^2]
        & = \int_1^2 \frac{\sin\Phi_a}{2\pi} \frac{2\pi}{\lambda_a \sin\Phi_a} \left[ \cos\left( \frac{2\pi}{\lambda_a} \right) + \lambda_a - 1 \right] \D{\lambda_a} \\
        & = \int_1^2 \frac{\lambda_a - 1 + \cos(2\pi/\lambda_a)}{\lambda_a} \D{\lambda_a} \\
        & = 1 - \log 2 + \int_{\pi}^{2\pi} \frac{\cos u}{u} \D u \approx 0.210624,
    \end{align*}
    where in the last line we used substitution $u = 2\pi/\lambda_a$.
    Using $\bbE[\sfA_{\tau_\bbD}] = \pi \bbE[M^2 - Y_T^2]$ we get the desired result.
\end{proof}


\section{Star hull and topological hull}\label{sec:star_topological}

In this last section we mention two different types of hulls, namely the star hull and topological hull. We explicitly compute the expected area of the star hull, and we provide a Monte Carlo estimate for the expected area of the topological hull

\subsection{Star hull}
Roughly speaking, the \emph{star hull} of a set $\cA\subset\mathbb{R}^2$, denoted by $\strh \cA$, is the smallest star-shaped set (with respect to the origin) that contains $\cA$. More precisely, we have
\begin{equation*}
\strh\cA=\{\rho\boldsymbol{x}:0\leq \rho\leq 1\text{ and }\boldsymbol{x}\in\cA\}.
\end{equation*}
The star hulls of planar Brownian motion and bridge run for unit time were recently studied in \cite{star_hull}, where exact formulas for their expected areas were computed. See Figure \ref{fig:brownian-hulls} for a depiction of the star hull of planar Brownian motion run until first exiting the disk.

For $t\geq 0$, let $\cH_t^\star=\strh \bfW_{[0,t]}$. In the following theorem, we compute $\bbE[\Area(\cH^\star_{\tau_\bbD})]$. The proof is essentially the same argument from Theorem \ref{thm:main} applied to the \emph{radial function} of the trace $\bfW_{[0,\tau_\bbD]}$ instead of the support function. For this reason, we combine all of the steps into one proof and omit some of the repetitive details.

\begin{theorem}\label{thm:star_hull}
The expected value of the area of the star hull of standard planar Brownian motion run until the exit time from the unit disk is given by
\[
\bbE[\Area(\cH^\star_{\tau_\bbD})]=\pi-\frac{8}{3}\approx 0.474925.
\]
\end{theorem}

\begin{proof}
For $\theta\in [0,2\pi)$, let $r(\theta)= \sup\{\rho \geq 0 : \rho\,\be_\theta \in \bfW_{[0,\tau_\bbD]}\}$ denote the radial function of the trace $\bfW_{[0,\tau_\bbD]}$. By \cite[Proposition 3]{star_hull}, we can represent the area of the star hull by 
\begin{equation}\label{eq:polar_formula_star_hull}
	\Area(\cH^\star_{\tau_\bbD})=\frac{1}{2}\int_0^{2\pi} r(\theta)^2\D{\theta}.
\end{equation}
Similarly to \eqref{eq:exp_per_via_exp_M}, applying Tonelli's theorem along with rotational invariance results in  
\begin{equation}\label{eq:star_area}
\bbE[\Area(\cH^\star_{\tau_\bbD})]=\pi\, \bbE\mleft[r(0)^2\mright].
\end{equation}

In order to compute the right-hand side of \eqref{eq:star_area}, we derive the distribution of $r(0)$ using the same method that was used for $M$. Towards this end, for $a\in [0,1)$, let $\cS_a$ denote the \emph{slit disk} formed by removing the line segment $[a,1)$ from $\bbD$. More precisely, 
\[
\cS_a=\bbD\setminus\{z\in \bbC:\Im z=0\text{ and }\Re z\geq a\}.
\]
The key observation is that $r(0)\geq a$ if and only if $\bfW$ hits the line segment $[a,1]$ by the time it exits $\bbD$. In terms of harmonic measure in the domain $\cS_a$ seen from $0$, we have
\begin{equation}\label{eq:radial_harmonic}
\bbP\big(r(0)\geq a\big)=\omega^0_{\cS_a}\big([a,1]\big).
\end{equation}

We use another M\"{o}bius transformation to map $\cS_a$ onto $\cS_0$. This target domain is a convenient choice since harmonic measure in $\cS_0$ has been studied before in the context of Milloux's problem and Beurling's projection estimate; see \cite[Section 3.2]{Ahlfors2} and \cite[Theorem 3.76]{Lawler}. Specifically, for $a\in[0,1)$, consider the M\"{o}bius transformation given by
\[
\mathfrak{m}_a(z)=\frac{z-a}{1-az}.
\]
This function maps $\bbD$ conformally onto itself; see \cite[Section 2.1]{Lawler}. Moreover, $\mathfrak{m}_a$ is also a linear fractional transformation, so it maps all of $\bbC$ conformally onto itself. It follows that $\mathfrak{m}_a$ is a continuous bijection of $\overline{\bbD}$ onto itself. In particular, it maps the segment $[a,1]$ onto the segment $[0,1]$, so that $\mathfrak{m}_a(\cS_a)=\cS_0$ with $\mathfrak{m}_a(0)=-a$. We can now use conformal invariance of harmonic measure along with reflection symmetry to rewrite \eqref{eq:radial_harmonic} as 
\begin{align}
\bbP\big(r(0)\geq a\big)=\omega^{-a}_{\cS_0}\big([0,1]\big)&=\omega^a_{-\cS_0}\big([-1,0]\big)\nonumber \\
&=1-\frac{4}{\pi}\arctan\sqrt{a},\label{eq:radial_distribution}
\end{align}
where the last equality is an easy consequence of conformal invariance, see e.g.\ \cite[page 42]{Ahlfors2}. 

It remains to use \eqref{eq:radial_distribution} to compute the expected value of $r(0)^2$. First note the identity
\[
\int_0^1 a\arctan\sqrt{a}\D{a}=\frac{1}{3},
\]
which is straightforward to deduce using the change of variables $a\mapsto t^2$ followed by integration by parts. We can now use this identity along with \eqref{eq:star_area} and \eqref{eq:radial_distribution} to write
\begin{align*}
\bbE[\Area(\cH^\star_{\tau_\bbD})]&=2\pi\int_0^1 a\,\bbP\big(r(0)\geq a\big)\D{a}\\
&=2\pi\int_0^1 a\D{a}-8\int_0^1 a\arctan\sqrt{a}\D{a}\\
&=\pi-\frac{8}{3}.
\end{align*}
\end{proof}


\subsection{Topological hull and simulations}

In addition to the already mentioned convex and star hulls of the trajectory of planar Brownian motion stopped upon exiting the unit disk, we mention here a third type of hull, the so-called \emph{topological hull}. The topological hull of the trajectory (also known as the Brownian hull) is the union of the trace with all of the bounded connected components of its complement in the plane (see Figure~\ref{fig:brownian-hulls}). The expected area of the topological hull of the standard planar Brownian bridge was computed exactly in \cite{Garban_Ferreras}, while that of planar Brownian motion run for unit time was investigated in \cite{star_hull}. We denote the topological hull of planar Brownian motion run until exiting the unit disk by $\cH^{\mathrm{T}}_{\tau_{\bbD}}$. It is straightforward to see that the three hulls in this paper satisfy the inclusion chain 
\begin{equation}\label{eq:inclusion_chain}
\cH^{\mathrm{T}}_{\tau_{\bbD}}\subset\cH^\star_{\tau_{\bbD}}\subset \cH_{\tau_{\bbD}};
\end{equation}
refer to \cite[Lemma 2]{star_hull} for a proof. See Figure \ref{fig:brownian-hulls} for illustrations of all three of these hulls drawn for the same planar Brownian motion trajectory.

\begin{figure}[htbp]
    \centering

    \begin{subfigure}[t]{0.48\textwidth}
        \centering
        \includegraphics[width=\linewidth]{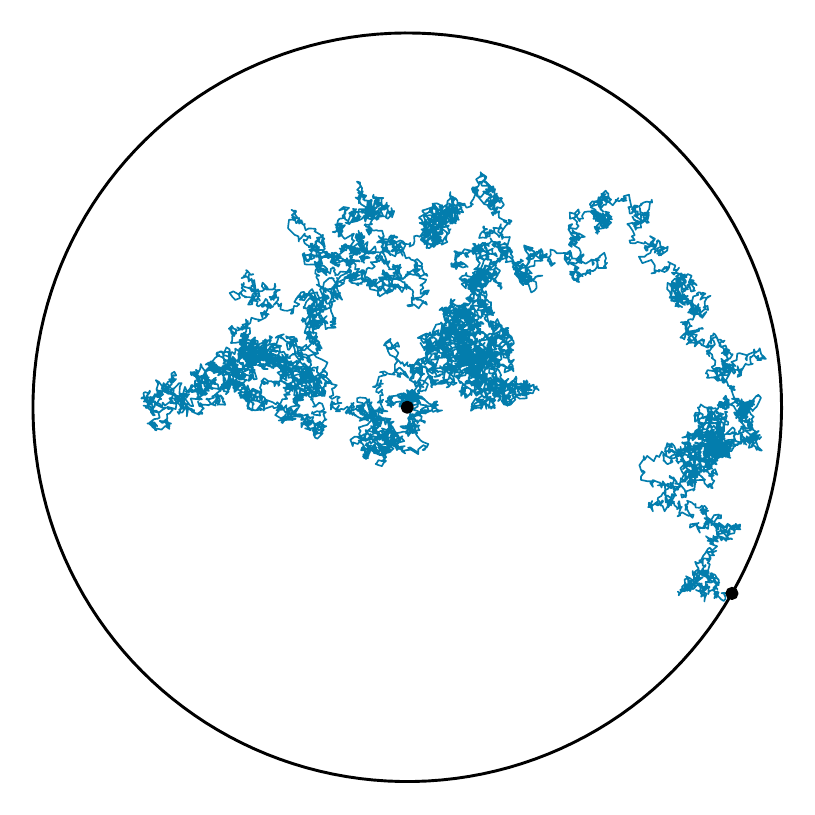}
        \caption{Brownian trajectory until exiting the disk.}
        \label{fig:brownian-path}
    \end{subfigure}
    \hfill
    \begin{subfigure}[t]{0.48\textwidth}
        \centering
        \includegraphics[width=\linewidth]{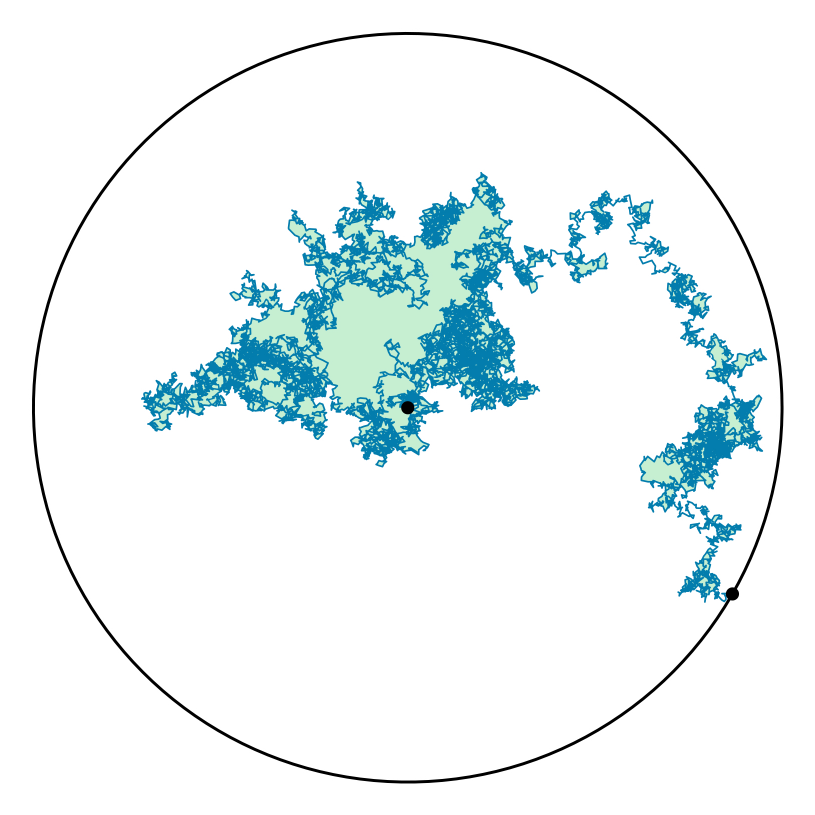}
        \caption{Topological hull of the same trajectory.}
        \label{fig:brownian-topological-hull}
    \end{subfigure}

    \vspace{0.8em}

    \begin{subfigure}[t]{0.48\textwidth}
        \centering
        \includegraphics[width=\linewidth]{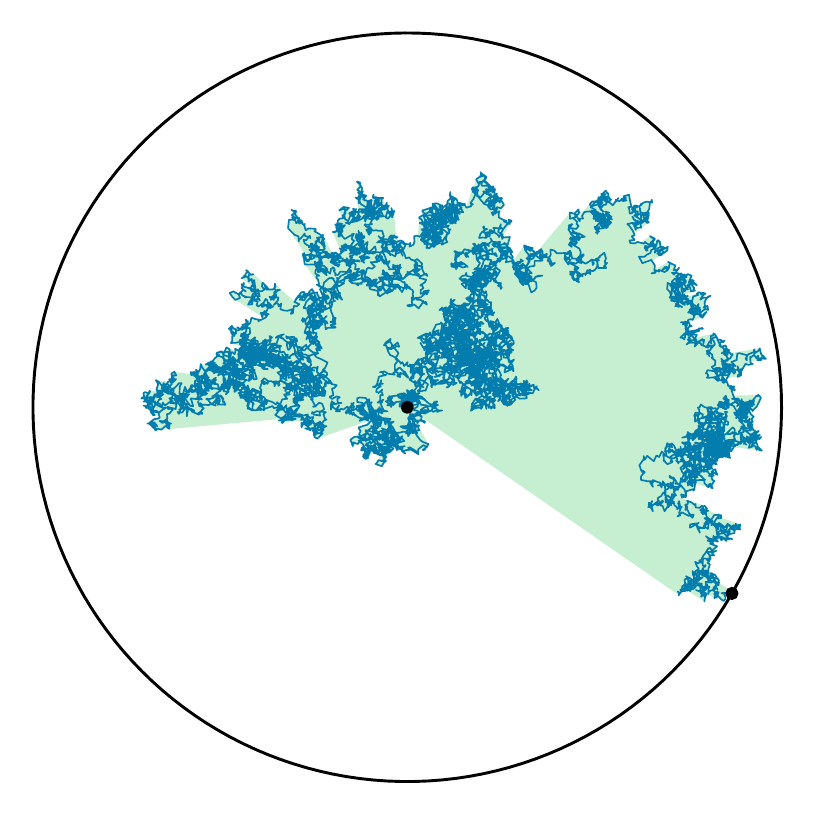}
        \caption{Star hull of the same trajectory.}
        \label{fig:brownian-star-hull}
    \end{subfigure}
    \hfill
    \begin{subfigure}[t]{0.48\textwidth}
        \centering
        \includegraphics[width=\linewidth]{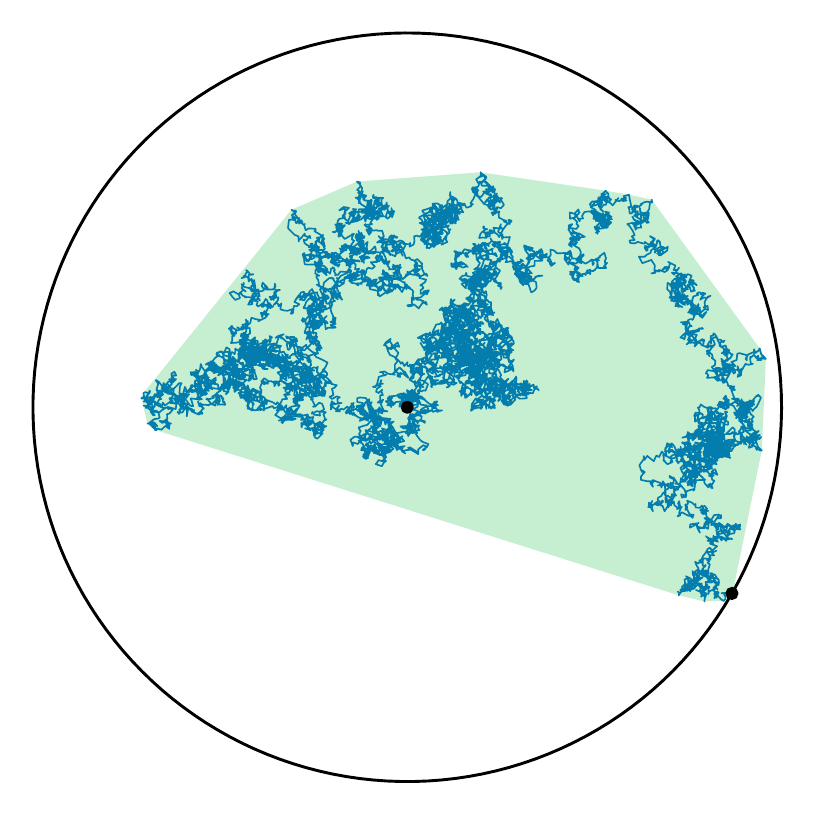}
        \caption{Convex hull of the same trajectory.}
        \label{fig:brownian-convex-hull}
    \end{subfigure}

    \caption{Illustrations of the same planar Brownian motion trajectory run until first exiting the unit disk, together with the three associated hulls.}
    \label{fig:brownian-hulls}
\end{figure}

For completeness, we ran simulations for all three types of hulls mentioned. Depending on the particular type of hull, we adjusted some of the path simulation parameters to account for the computational demands of the hull algorithm.



\subsubsection{Convex hull}\label{sec:conv_hull_area}
We simulated $10^5$ trajectories of the standard planar Brownian motion until it exits the unit disk $\bbD$ for the first time. The simulated trajectories had time increments of size $10^{-7}$, and steps were Gaussian (with mean zero, and variance $10^{-7}$). For each realization of the trajectory we found its convex hull, and computed the perimeter and the area of that convex hull. Since the function \texttt{scipy.spatial.ConvexHull} returns extremal points of the convex hull in counterclockwise order, computation of the perimeter is trivial, and for the area we use the shoelace formula. Moreover, for each realization of the trajectory we found the point at which the maximal displacement in the positive direction of the $x$-axis is attained, and we saved the $y$-coordinate of that point. This enabled us to obtain a Monte Carlo approximation of the quantity $\bbE[h'(0)^2]$ and, through \eqref{eq:area_via_h0}, to give an alternative Monte Carlo approximation of $\bbE[\sfA_{\tau_{\bbD}}]$. 

The Monte Carlo estimate of the expected perimeter was $3.2136$. Comparing it to the numerical value of $3.2148$ from Theorem \ref{thm:main}, we see that the error is of size $10^{-3}$. As explained, we obtained a Monte Carlo estimate of the expected area in two different ways. Namely, by computing the area directly, and through the value of the $y$-coordinate of the point which corresponds to the maximal displacement of the simulated trajectory in the positive direction of the $x$-axis. The direct computation gives a Monte Carlo estimate of $0.6612$, and the alternative one gives $0.6618$. Notice that here the error is even smaller, namely $10^{-4}$, since the numerical value (computed in Theorem~\ref{thm:area}) is 0.6617.

\subsubsection{Star hull}
To obtain a Monte Carlo estimate of the expected area of the star hull, we rely on the standard polar formula \eqref{eq:polar_formula_star_hull}. In the simulations, the Brownian path is approximated by a polygonal trajectory $w_0, w_1, \ldots, w_N$ obtained from a time discretization with step $\Delta t = 10^{-6}$, and taking Gaussian steps with expectation zero, and variance equal to $\Delta t$. The interval $[0, 2\pi)$ is divided into $m = 2000$ equally spaced angles $\theta_k = 2\pi k / m$. For each direction $\be_{\theta_k}$, the value $r(\theta_k)$ is computed as the maximal distance at which the ray $\{\rho\be_{\theta_k} : \rho \ge 0\}$ intersects the star hull of the polygonal path. Since the latter equals the union of the triangles $\conv\{0, w_i, w_{i+1}\}$, this reduces to intersecting the ray with each segment $[w_i, w_{i+1}]$. Concretely, for each segment we solve
\begin{equation*}
	\rho\be_{\theta_k} = (1-s)w_i + sw_{i+1},
\end{equation*}
with $\rho \ge 0$ and $s \in [0, 1]$. Admissible solutions correspond to intersections with the edge of the triangle, and the largest such $\rho$ over all segments yields $r(\theta_k)$. The area is then approximated by the Riemann sum
\begin{equation*}
	\frac{1}{2} \sum_{k = 0}^{m-1}r(\theta_k)^2 \Delta \theta, \qquad \Delta \theta = 2\pi/m.
\end{equation*}
A Monte Carlo estimate of the expected star hull area is obtained by averaging this over $10^5$ independent simulated trajectories. The obtained Monte Carlo estimate is $0.4725$ which is at distance of order $10^{-3}$ from the value $0.4749$ obtained in Theorem~\ref{thm:star_hull}.

\subsubsection{Topological hull}\label{sec:top_hull}

Unlike the convex and star hulls, the topological hull was simulated using a random walk on the square lattice $\mathbb{Z}^2$ stopped when its distance from the origin exceeded $10^3$. Then the algorithm from \cite{Richard} was used to determine the outer boundary of the topological hull of this lattice path, from which the integer-valued area was calculated using the shoelace formula. Lastly, this area was scaled by $10^{-6}$ to approximate the area of the Brownian topological hull. This procedure was repeated $10^5$ times and averaged to provide a Monte Carlo estimate of $0.2816$. While we don't have an exact expression for the expected area like we do with the other hulls, the inclusion chain \eqref{eq:inclusion_chain} implies that the expected area of the star hull provides an upper bound for the expected area of the topological hull.


\section*{Acknowledgments} 

\noindent
We thank Krunoslav Ivanovi\'{c} for fruitful discussions.

Financial support through the \emph{Croatian Science Foundation} under projects IP-2022-10-2277 (for S.\ \v{S}ebek) and IP-2022-10-5116 (for R.\ Mrazovi\'{c}) is gratefully acknowledged. This research was also funded by the European union--NextGenerationEU through the National Recovery and Resilience Plan 2021-2026 Institutional grant of University of Zagreb Faculty of Electrical Engineering and Computing (VALOR). This work was carried out within a project DIGIT.2.1.02.016 funded by the Digital, Innovation, and Green Technology Project – DIGIT Project (IBRD Loan No. 9558‑HR).

\bibliographystyle{abbrv}
\bibliography{CH_until_ET}

\begin{thebibliography}{10}

\bibitem{Ahlfors2}
L.~V. Ahlfors.
\newblock {\em Conformal invariants}.
\newblock AMS Chelsea Publishing, Providence, RI, 2010.
\newblock Topics in geometric function theory, Reprint of the 1973 original,
  With a foreword by Peter Duren, F. W. Gehring and Brad Osgood.

\bibitem{cauchy1832memoire}
A.~L. Cauchy.
\newblock {\em M\'{e}moire sur la rectification des courbes et la quadrature
  des surfaces courbes}.
\newblock Lith. de C. Mantoux, 1832.

\bibitem{CPS}
W.~Cygan, H.~Panzo, and S.~\v{S}ebek.
\newblock Bounds on the size of the convex hull of planar {B}rownian motion and
  related inverse processes.
\newblock {\em J. Korean Math. Soc.}, 62(5):1265--1295, 2025.

\bibitem{DeBruyneBenichouMajumdarSchehr}
B.~De~Bruyne, O.~B\'{e}nichou, S.~N. Majumdar, and G.~Schehr.
\newblock Statistics of the maximum and the convex hull of a {B}rownian motion
  in confined geometries.
\newblock {\em J. Phys. A}, 55(14):Paper No. 144002, 17, 2022.

\bibitem{DLMF}
{\it NIST Digital Library of Mathematical Functions}.
\newblock \url{https://dlmf.nist.gov/}, Release 1.2.6 of 2026-03-15.
\newblock F.~W.~J. Olver, A.~B. {Olde Daalhuis}, D.~W. Lozier, B.~I. Schneider,
  R.~F. Boisvert, C.~W. Clark, B.~R. Miller, B.~V. Saunders, H.~S. Cohl, and
  M.~A. McClain, eds.

\bibitem{ElBachir_PhD_thesis}
M.~El~Bachir.
\newblock {\em L'enveloppe convex du mouvement {B}rownien}.
\newblock Doctoral thesis, Universit\'{e} Toulouse III–Paul Sabatier, 1983.

\bibitem{Eldan}
R.~Eldan.
\newblock Volumetric properties of the convex hull of an {$n$}-dimensional
  {B}rownian motion.
\newblock {\em Electron. J. Probab.}, 19:no. 45, 34, 2014.

\bibitem{Garban_Ferreras}
C.~Garban and J.~A. Trujillo~Ferreras.
\newblock The expected area of the filled planar {B}rownian loop is {$\pi/5$}.
\newblock {\em Comm. Math. Phys.}, 264(3):797--810, 2006.

\bibitem{goldman}
A.~Goldman.
\newblock Le spectre de certaines mosa\"{i}ques poissoniennes du plan et
  l'enveloppe convexe du pont brownien.
\newblock {\em Probab. Theory Related Fields}, 105(1):57--83, 1996.

\bibitem{Gradshteyn_Ryzhik}
I.~S. Gradshteyn and I.~M. Ryzhik.
\newblock {\em Table of integrals, series, and products}.
\newblock Elsevier/Academic Press, Amsterdam, eighth edition, 2015.
\newblock Translated from the Russian, Translation edited and with a preface by
  Daniel Zwillinger and Victor Moll.

\bibitem{HaasMallein}
B.~Haas and B.~Mallein.
\newblock Fragmentation processes and the convex hull of the {B}rownian motion
  in the disk.
\newblock {\em Ann. H. Lebesgue}, 8:219--253, 2025.

\bibitem{Hsiung}
C.~C. Hsiung.
\newblock {\em A first course in differential geometry}.
\newblock Pure and Applied Mathematics. John Wiley \& Sons, Inc., New York,
  1981.
\newblock A Wiley-Interscience Publication.

\bibitem{Jovalekic}
M.~Jovaleki\'{c}.
\newblock Lower bound for the diameter of planar {B}rownian motion.
\newblock {\em Bull. Math. Soc. Sci. Math. Roumanie (N.S.)},
  64(112)(3):281--284, 2021.

\bibitem{Kabluchko-Zapor-TAMS}
Z.~Kabluchko and D.~Zaporozhets.
\newblock Intrinsic volumes of {S}obolev balls with applications to {B}rownian
  convex hulls.
\newblock {\em Trans. Amer. Math. Soc.}, 368(12):8873--8899, 2016.

\bibitem{Lawler}
G.~F. Lawler.
\newblock {\em Conformally invariant processes in the plane}, volume 114 of
  {\em Mathematical Surveys and Monographs}.
\newblock American Mathematical Society, Providence, RI, 2005.

\bibitem{Letac}
G.~Letac.
\newblock Advanced problem 6230.
\newblock {\em Amer. Math. Monthly}, 85(3):686, 1978.

\bibitem{Majumdar}
S.~N. Majumdar, A.~Comtet, and J.~Randon-Furling.
\newblock Random convex hulls and extreme value statistics.
\newblock {\em J. Stat. Phys.}, 138(6):955--1009, 2010.

\bibitem{MX}
J.~McRedmond and C.~Xu.
\newblock On the expected diameter of planar {B}rownian motion.
\newblock {\em Statist. Probab. Lett.}, 130:1--4, 2017.

\bibitem{Metzler}
A.~Metzler.
\newblock On the first passage problem for correlated {B}rownian motion.
\newblock {\em Statist. Probab. Lett.}, 80(5-6):277--284, 2010.

\bibitem{Molchanov-Wespi}
I.~Molchanov and F.~Wespi.
\newblock Convex hulls of {L}\'evy processes.
\newblock {\em Electron. Commun. Probab.}, 21:Paper No. 69, 11, 2016.

\bibitem{MortersPeres}
P.~M\"{o}rters and Y.~Peres.
\newblock {\em Brownian motion}, volume~30 of {\em Cambridge Series in
  Statistical and Probabilistic Mathematics}.
\newblock Cambridge University Press, Cambridge, 2010.
\newblock With an appendix by Oded Schramm and Wendelin Werner.

\bibitem{star_hull}
H.~Panzo.
\newblock Expected area of the star hull of planar {B}rownian motion and
  bridge.
\newblock arXiv:2602.10974, 2026.

\bibitem{hi_dim_hulls}
H.~Panzo and E.~Socher.
\newblock Bounds on some geometric functionals of high dimensional {B}rownian
  convex hulls and their inverse processes.
\newblock {\em Canad. Math. Bull.}, 69(1):222--235, 2026.

\bibitem{Richard}
C.~Richard.
\newblock Area distribution of the planar random loop boundary.
\newblock {\em J. Phys. A}, 37(16):4493--4500, 2004.

\bibitem{sebek}
S.~{\v{S}}ebek.
\newblock Convex hull of {B}rownian motion and {B}rownian bridge.
\newblock {\em Markov Process. Related Fields}, 30(4):459--475, 2024.

\bibitem{Takacs}
L.~Tak\'{a}cs.
\newblock Expected perimeter length.
\newblock {\em Amer. Math. Monthly}, 87(2):142, 1980.

\bibitem{TsukermanVeomett}
E.~Tsukerman and E.~Veomett.
\newblock Brunn-{M}inkowski theory and {C}auchy's surface area formula.
\newblock {\em Amer. Math. Monthly}, 124(10):922--929, 2017.

\end{thebibliography}

\end{document}